\documentclass[11pt, reqno]{amsart}
\usepackage{hyperref}
  \evensidemargin
\oddsidemargin
\usepackage{mathtools,amsmath,amssymb,amsthm,mathrsfs,calc,graphicx,xcolor,cleveref,dsfont,tikz,pgfplots,bm}
\usepackage[british]{babel}

\usepackage{amsfonts}              
\usepackage[T1]{fontenc}
\numberwithin{equation}{section}
\numberwithin{figure}{section}

\newcommand{\mfm}{\mathfrak{m}}

\theoremstyle{plain}
\newtheorem{theorem}{Theorem}[section]
\newtheorem{lemma}[theorem]{Lemma}
\newtheorem{corollary}[theorem]{Corollary}
\newtheorem{proposition}[theorem]{Proposition}

\theoremstyle{definition}

\theoremstyle{remark}
\newtheorem{remark}[theorem]{Remark}

\newcommand{\pos}{\mathop{\mathrm{pos}}\nolimits}
\newcommand{\sgn}{\mathop{\mathrm{sgn}}\nolimits}

\def\EE{\mathbb{E}}

\def\NN{\mathbb{N}}

\def\PP{\mathbb{P}}

\def\RR{\mathbb{R}}
\def\SS{\mathbb{S}}

\def\bP{\mathbf{P}}

\def\cA{\mathcal{A}}

\def\sN{\mathscr{N}}

\def\sN{\mathscr{N}}

\def\sV{\mathscr{V}}

\newcommand{\eqdistr}{\stackrel{d}{=}}

\newcommand{\ii}{{\rm{i}}}
\newcommand{\Mod}[1]{\ (\mathrm{mod}\ #1)}

\newcommand{\E}{\mathbb E}

\newcommand{\R}{\mathbb{R}}
\newcommand{\N}{\mathbb{N}}

\renewcommand{\P}{\mathbb{P}}

\def\dint{\textup{d}}

\def\SO{\textup{SO}}

\begin{document}

\title{The typical cell of a Voronoi tessellation on the sphere}

\author{Zakhar Kabluchko}
\address{Zakhar Kabluchko: Institut f\"ur Mathematische Stochastik,
Westf\"alische Wilhelms-Universit\"at M\"unster,
Orl\'eans-Ring 10,
48149 M\"unster, Germany}
\email{zakhar.kabluchko@uni-muenster.de}

\author{Christoph Th\"ale}
\address{Christoph Th\"ale: Fakult\"at f\"ur Mathematik,
Ruhr-Universit\"at Bochum,
44780 Bochum, Germany}
\email{christoph.thaele@rub.de}

\date{}

\begin{abstract}
\noindent  The typical cell of a Voronoi tessellation generated by $n+1$ uniformly distributed random points on the $d$-dimensional unit sphere $\mathbb S^d$ is studied. Its $f$-vector is identified in distribution with the $f$-vector of a beta' polytope generated by $n$ random points in $\mathbb R^d$. Explicit formulae for the expected $f$-vector are provided for any $d$ and the low-dimensional cases $d\in\{2,3,4\}$ are studied separately. {This implies an explicit formula for the total number of $k$-dimensional faces in the spherical Voronoi tessellation as well.}
\end{abstract}



\maketitle



\section{Introduction}

Let $E$ be a metric space and $\{x_i:i\in I\}$ a finite (or, more generally, locally finite) collection of points in $E$, where $I$ is some index set. The \textit{Voronoi cell} of a point $x_i$ is the set of all points in $E$ whose distance to $x_i$ is not greater than the distance to any other point $x_j$ with $i\neq j$. The \textit{Voronoi tessellation} or Voronoi diagram associated with the set $\{x_i:i\in I\}$ is then just the collection of all such Voronoi cells. The study of Voronoi tessellations has attracted a lot of attention in computational as well as in stochastic geometry. To a great extent this is because of their various applications ranging from the modelling of biological tissues or polycrystalline microstructures in metallic alloys to classification problems in machine learning. We refer the reader to the monographs \cite{DeBerg,Mitchell,MollerVoronoi,Okabe} for details and many more references.

In this note we consider Voronoi tessellations of the unit sphere that are generated by a (finite) collection of uniformly distributed, independent random points. Unlike their Euclidean counterparts, for which there exists an extensive literature (see \cite{MollerVoronoi,Okabe,SW,SKM} and the references cited therein), the mathematical properties of spherical Voronoi tessellations are only poorly understood. Just a few results for Voronoi tessellation on the $2$-dimensional unit sphere are available in the classical reference \cite{MilesSphere}. On the other hand, Voronoi tessellations induced by points on a general manifold become increasingly important in computational geometry, see \cite{ChengBOOK,DeBerg}. Our goal is to partially fill the resulting gap by considering the  combinatorial structure of what is called the typical cell of a Voronoi tessellation on the $d$-dimensional unit sphere for general $d\geq 2$. More precisely, we shall study the $f$-vector of the typical spherical Voronoi cell. We do this by establishing and exploiting a new connection of such typical Voronoi cells with the classes of random beta and beta' polytopes. These have recently been under intensive investigation~\cite{Bonnet1,Bonnet2,Chasapis,GroteKabluchkoThaele,kabluchko_angles,kabluchko_formula,kabluchko_poisson_zero,KabluchkoMarynychTemesThaele,KabluchkoTemesvariThaele,beta_polytopes}. In fact, as it will turn out, the $f$-vector of the typical spherical Voronoi cell can be identified in distribution with the $f$-vector of (the dual of) a particular random beta' polytope. Also the explicit expected values can be determined from this distributional identity and the known results for beta' polytopes. We establish in addition a link between the expected $f$-vector of typical spherical Voronoi cells and that of a special beta polytope. Of special interest are the low-dimensional cases $d\in\{2,3,4\}$ which will be examined separately.

We would like to point out that our paper continues a recent line of research in stochastic geometry which focuses on the study of non-Euclidean geometric random structures. As examples we mention the studies of random convex hulls in spherical convex bodies or on half-spheres \cite{BaranyHugReitznerSchneider,BesauThaele,KabluchkoMarynychTemesThaele,kabluchko_poisson_zero}, the results on random tessellations by great hyperspheres \cite{ArbeiterZaehleMosaics,HugReichenbacher,HugSchneider2016,MilesSphere}, the central and non-central limit theorems for Poisson hyperplanes in hyperbolic spaces \cite{HeroldHug}, the papers \cite{DeussHoerrmannThaele,HugThaele} on splitting tessellations on the sphere, the asymptotic investigation of Voronoi tessellations on general Riemannian manifolds \cite{Calkaetc} and the general limit theory for stabilizing functionals of point processes in manifolds \cite{PenroseYukichMf}.

\begin{figure}[t]\label{fig:tess}
\begin{center}
\includegraphics[trim = 20mm 20mm 20mm 20mm, clip, width=0.4\columnwidth]{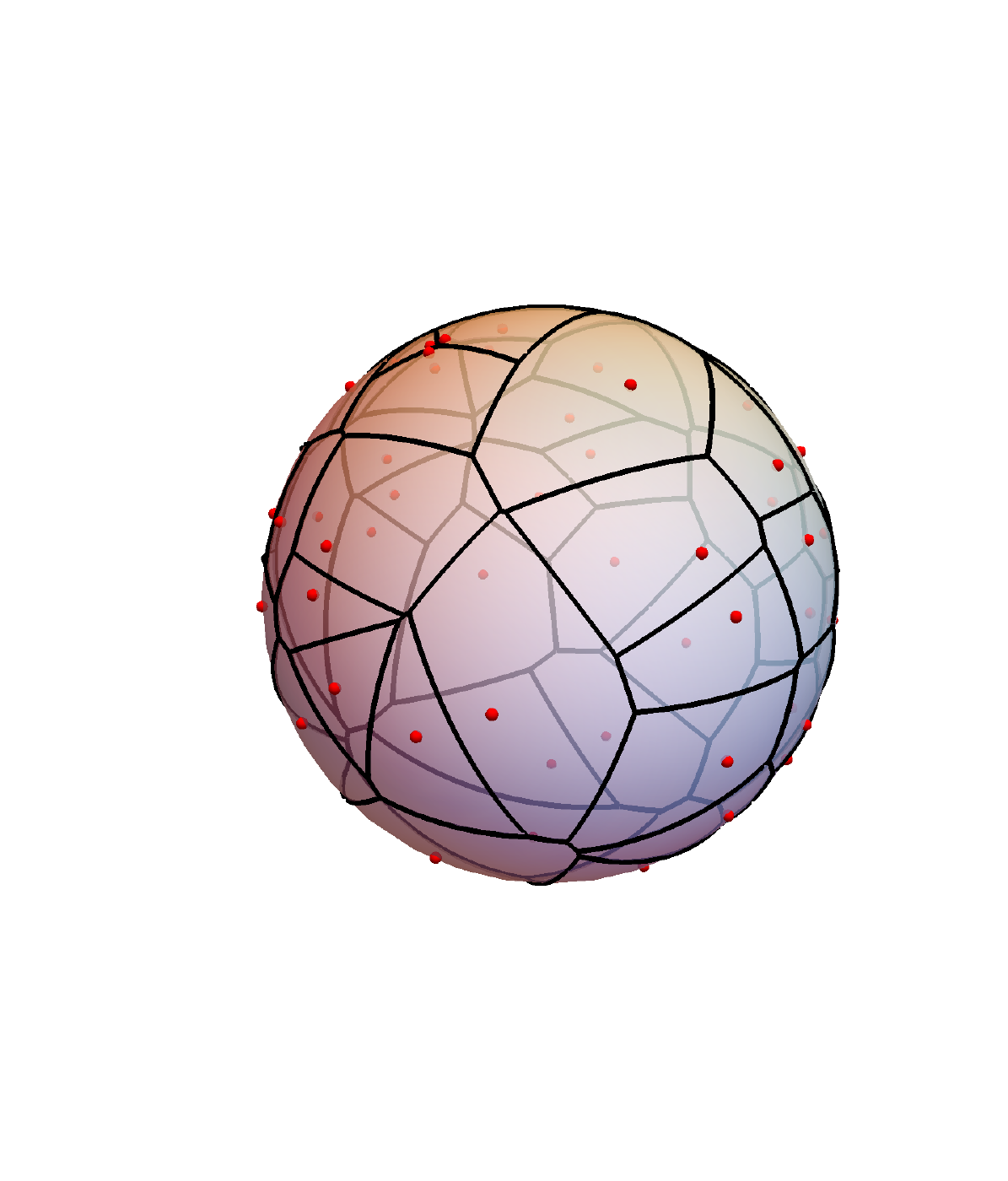}\quad
\includegraphics[trim = 20mm 24mm 20mm 20mm, clip, width=0.4\columnwidth]{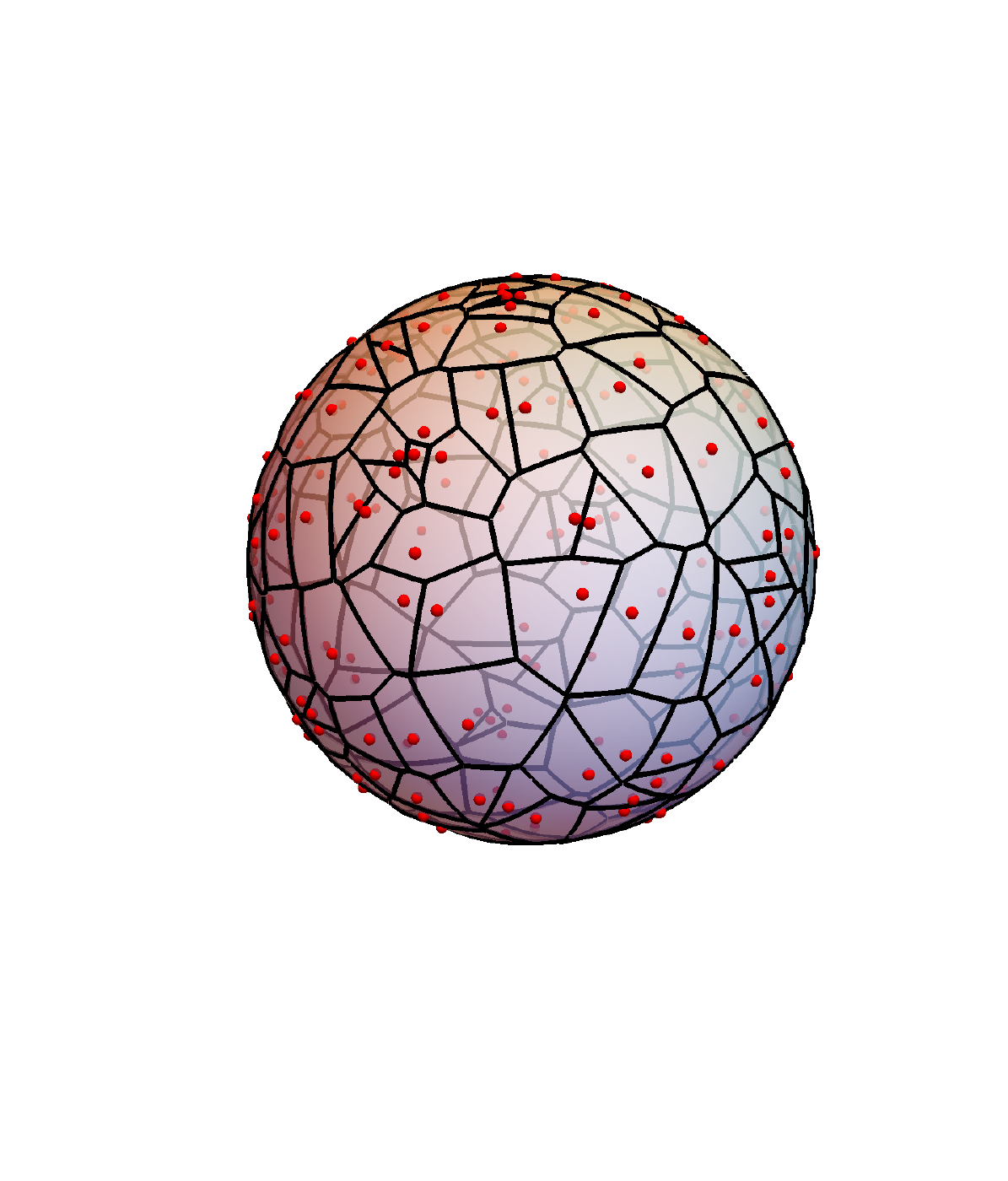}
\end{center}
\caption{Simulations of spherical Voronoi tessellations on $\SS^2$ with  $50$ cells (left) and $200$ cells (right).}
\end{figure}

\section{The typical Voronoi cell and its \texorpdfstring{$f$}{f}-vector}

\subsection{The typical Voronoi cell}\label{subsec:TypicalVoronoiCell}

We are now going to introduce our framework. Let $\SS^d$ be the $d$-dimensional unit sphere, which we think of being embedded in $\RR^{d+1}$ in such a way that it is centred at the origin of $\RR^{d+1}$. A {generic} point in $\R^{d+1}$ is denoted by $x = (x_0,x_1,\ldots,x_d)$. The dimension of the sphere, $d\in\N$,  is fixed once and for all. The normalized spherical Lebesgue measure on $\SS^d$ is denoted by $\sigma_d$. Let $X_1,\ldots,X_n$ be $n\in\NN$ independent random points sampled on $\SS^d$ according to $\sigma_d$ and defined over some underlying probability space $(\Omega,\cA,\PP)$.  The \textit{binomial process} $\xi_{n}:=\{X_1,\ldots,X_n\}$ is the point process on $\SS^d$ with atoms at $X_1,\ldots,X_n$.  We can now construct the \textit{spherical Voronoi tessellation} based on $\xi_{n}$ as follows. If $\rho(\,\cdot\,,\,\cdot\,)$ denotes the geodesic distance on $\SS^d$, we let $C_{i,n}$ be the \textit{Voronoi cell} of a point $X_i\in\xi_{n}$, that is,
$$
C_{i,n} := \{z\in\SS^d:\rho(X_i,z)\leq \rho(X_j,z)\text{ for all } j\in\{1,\ldots,n\}\},
\qquad
i\in \{1,\ldots,n\}.
$$
As in the Euclidean case (see \cite[Chapter 10]{SW}), one shows that the sets $C_{1,n},\ldots,C_{n,n}$ are in fact spherical polytopes covering $\SS^d$ and having disjoint interiors. Here, we recall that a spherical polytope is defined as an intersection of $\SS^d$ and a polyhedral convex cone and that the latter is defined as an intersection of finitely many {closed} half-spaces whose bounding hyperplanes contain the origin.
The collection $\{C_{1,n},\ldots,C_{n,n}\}$ of all Voronoi cells of points of $\xi_{n}$ is what we call the \textit{spherical Voronoi tessellation} $\mfm_{n,d}$, see Figure~\ref{fig:tess} for two sample realizations.

In this note we are interested in the \textit{typical cell} of such a spherical Voronoi tessellation. 
Roughly speaking, the typical cell arises by picking one of the cells $C_{i,n}$ uniformly at random and rotating it such that its ``{centre}'' $X_i$ becomes the north pole $e:=(1,0,\ldots,0)$ of $\SS^d$.
To make this  precise, let {$N=N_{n}$} be a random variable with uniform distribution on the set $\{1,\ldots,n\}$ and assume that  {$N$} is independent of the binomial process $\xi_{n}$.
Also, for every point $v\in \SS^d$ we fix some orthogonal transformation $O_v:\R^{d+1} \to\R^{d+1}$ such that $O_v v = e$ and assume that the matrix elements of $O_v$ are Borel functions of $v$. Then, the typical cell of the Voronoi tessellation $\mfm_{n, d}$ is a random spherical polytope $\sV_{n,d}$  defined by
\begin{equation}\label{eq:def_typical_voronoi}
 {\sV_{n,d} := O_{X_N} C_{N,n}.}
\end{equation}
Since $X_1,\ldots,X_n$ are exchangeable, the tuple  {$(\xi_n, X_N)$} has the same joint law as $(\xi_n,X_1)$ and we arrive at the following distributional equality:
$$
\sV_{n,d} \eqdistr  O_{X_1} C_{1,n}.
$$

In the following, it will be more convenient to consider   {a binomial process} with $n+1$ rather than with $n$ points. The next proposition states that the typical Voronoi cell $\sV_{n+1,d}$ of the binomial process $\xi_{n+1}$ has the same distribution as the Voronoi cell of the north pole $e$ in the point process $\xi_n\cup\{e\}$. Note that it also proves that the distribution of the typical cell does not depend on the choice of the family of orthogonal transformations $(O_v)_{v\in\SS^d}$.
\begin{proposition}\label{prop:typical_voronoi_rep}
We have the distributional equality
\begin{equation}\label{eq:def_typical_voronoi2}
\sV_{n+1,d} \eqdistr  \{z\in\SS^d: \rho(e,z)\leq \rho(X_j,z) \text{ for all } j\in\{1,\ldots,n\}\}.
\end{equation}
\end{proposition}
\begin{proof}
 Conditioning on $X_1=v$ and integrating over all $v\in\SS^d$, we can write the distribution of $\sV_{n+1,d}$ as follows:
\begin{align*}
\P[\sV_{n+1,d} \in B] = \int_{\SS^d}\P[O_{X_1} C_{1,n+1} \in B | X_1 = v]\, \sigma_d(\dint v),
\end{align*}
for every Borel set $B$ in the space of compact subsets of $\SS^d$ endowed with the usual Hausdorff distance. Recalling the definition of $C_{1,n+1}$, we can write
\begin{align*}
\P[O_{X_1} C_{1,n+1} \in B | X_1 = v]
&=
\P\left[O_{v} \left\{z\in\SS^d:\rho(v,z)\leq \min_{j=2,\ldots,n+1}\rho(X_j,z)\right\} \in B\right]\\
&=
\P\left[\left\{y\in\SS^d:\rho(v,O_v^{-1} y)\leq \min_{j=2,\ldots,n+1}\rho(X_j,O_v^{-1} y)\right\} \in B\right]\\
&=
\P\left[\left\{y\in\SS^d:\rho(e,y) \leq \min_{j=2,\ldots,n+1}\rho(O_v X_j,y)\right\} \in B\right]\\
&=
\P\left[\left\{y\in\SS^d:\rho(e,y) \leq \min_{j=1,\ldots,n}\rho(X_j,y)\right\} \in B\right],
\end{align*}
where we defined $y:= O_v z$ and used that $(O_v X_2,\ldots,O_v X_{n+1})$ has the same  {joint} law as $(X_1,\ldots,X_{n})$. Since the right-hand side does not depend on $v\in\SS^d$, we arrive at
$$
\P[\sV_{n+1,d} \in B]   = \P\left[\left\{y\in\SS^d:\rho(e,y) \leq \min_{j=1,\ldots,n}\rho(X_j,y)\right\} \in B\right],
$$
which completes the proof.
\end{proof}

For stationary tessellations in the Euclidean space $\R^d$, where the number of cells is almost surely infinite, one usually defines the typical cell using the concept of Palm distribution, which is a common device in stochastic geometry~\cite{SW}. The Palm approach can be applied on the sphere, too.   {Following \cite{RotherZaehle}, the Palm distribution $\bP_{\xi_{n+1}}^e$ of the binomial process $\xi_{n+1}$ with respect to a fixed point on the sphere (which we choose to be the north pole $e$) can formally be defined as follows. For $v\in\SS^d$ we let  $\Theta_v$ denote the set of all orientation-preserving orthogonal transformations $O:\RR^{d+1}\to\RR^{d+1}$ such that $Ov=e$. Note that $\Theta_e$ is a group which can be identified with $\text{SO}(d)$.  By $\nu_e$ we denote the unique Haar probability measure on $\Theta_e$ and define the image measure $\nu_v(A) := \nu_e (\{O O_v^{-1}: O\in A\})$, $A\subset \Theta_v$, on $\Theta_v$, where $O_v\in\Theta_v$ is arbitrary (in fact, the definition is independent of the choice of $O_v$, see~\cite{RotherZaehle}). The \textit{Palm distribution} $\bP_{\xi_{n+1}}^e$ with respect to the point $e$ is given by
\begin{align*}
	\bP_{\xi_{n+1}}^e(\,\cdot\,)
	&:=\frac{1}{n+1}\E\sum_{v\in\xi_{n+1}}\int_{\Theta_v}{\bf 1}(O^{-1} \xi_{n+1}\in\,\cdot\,)\,\nu_v(\dint O).
\end{align*}
From \cite[Lemma 6.14]{Kallenberg2} it is known that
$$
\bP_{\xi_{n+1}}^e(\,\cdot\,) = \bP_{\xi_n}(\xi_{n}\cup\{e\}\in\,\cdot\,),
$$
where $\bP_{\xi_{n}}$ denotes the distribution of the binomial process $\xi_{n}$. This is the analogue for binomial processes of the celebrated Slivnyak-Mecke theorem for Poisson processes \cite[Lemma 6.15]{Kallenberg2}.
In particular, it shows that the definition of the typical cell given above coincides with the definition based on the Palm approach.}

\subsection{Total number of faces}
Our goal is to describe the $f$-vector of the typical Voronoi cell $\sV_{n,d}$. More precisely, consider a  spherical polytope $P\subset\SS^d$ represented as an intersection of $\SS^d$ and a polyhedral convex cone $C$. The $k$-dimensional faces of $P$ are defined as intersections of $(k+1)$-dimensional faces of $C$ with $\SS^d$, where $k\in\{0,1,\ldots,d\}$. We denote by $\mathcal F_k(P)$ the set of $k$-dimensional faces of $P$ and by $f_k(P):= |\mathcal F_k(P)|$ their number.\footnote{If $P$ is degenerate (that is, if it contains a pair of diametrally opposite points), the above definitions may lead to results which look unnatural. For example, if $C$ is a half-plane, $d=1$, and $P$ is a semicircle, then $C$ has $1$ one-dimensional face and hence $f_0(P) = 1$ (rather than $2$, which seems more natural). In the following, the reader may assume that {$n\geq d+2$,} wh1ich implies that the typical Voronoi cell $\sV_{n,d}$ is non-degenerate  and these difficulties disappear. 
{  
Another possibility is to consider conical tessellations instead of the spherical ones.
}
}
Here, $|A|$ stands for the number of elements of a set $A$. The $d$-dimensional vector $\big(f_0(P),f_1(P),\ldots,f_{d-1}(P)\big)$ is called the \textit{$f$-vector} of $P$.

Before stating the results on the expected $f$-vector of the typical Voronoi cell, let us point out its connection to another natural quantity. The total number of $k$-dimensional faces of the tessellation $\mfm_{n, d}$ is denoted by
$$
f_k(\mfm_{n, d}) := \Bigg|\bigcup_{i=1}^n \mathcal F_k(C_{i,n})\Bigg|, \qquad k\in \{0,1,\ldots,d\}.
$$
Note that even if some face $F$ belongs to more than one cell $C_{i,n}$, it is counted only once in the above definition.
\begin{proposition}\label{prop:total_cells}
For all  {$n\geq d+1$} and $k\in \{0,\ldots,d\}$, we have
$$
\EE f_k (\mfm_{n, d}) = \frac{n}{d-k+1}\,\EE f_k(\sV_{n,d}).
$$
\end{proposition}
\begin{proof}
We use a double-counting argument. Let  {$M := \sum_{i=1}^{n} f_k(C_{i,n})$} be the number of pairs $(C_{i,n},F)$, where $C_{i,n}$ is a cell of the tessellation $\mfm_{n,d}$, and $F\subset C_{i,n}$ a $k$-dimensional face of $C_{i,n}$. On the one hand, the above definition~\eqref{eq:def_typical_voronoi} of the typical cell implies that
$$
\EE f_k(\sV_{n,d})
=
\EE f_k(O_{X_N} C_{N,n})
=
\EE f_k(C_{N,n})
=
\frac{1}{n} \sum_{i=1}^{n} \EE f_k(C_{i,n})
=
\frac{1}{n} \EE \sum_{i=1}^{n} f_k(C_{i,n})
=
\frac{\E M}{n}.
$$
On the other hand, the spherical Voronoi tessellation is normal, that is, every $k$-dimensional face belongs to $(d-k+1)$ cells of dimension $d$, with probability one (cf.\ Theorem~10.2.3 in~\cite{SW} for a similar statement in the Euclidean case).  It follows that almost surely
$$
M = (d-k+1) f_k(\mfm_{n,d}).
$$
By taking the expectations and comparing both identities, we arrive at the claim.
\end{proof}

\subsection{Reduction to beta' polytopes}

As anticipated above, our goal will be to identify the expected $f$-vector of the typical Voronoi cell $\sV_{n+1,d}$ generated by $n+1$ uniformly distributed random points on the $d$-dimensional unit sphere. We do this first in terms of the $f$-vector of random beta' polytopes, a notion we are going to explain next. For $\beta>d/2$ we define the probability density $\tilde f_{d, \beta}$ on $\R^d$ by
\begin{equation}\label{eq:BetaPrimeDensity}
\tilde f_{d, \beta}(x) := \tilde c_{d,\beta}\,(1+\|x\|^2)^{-\beta},\qquad\qquad
\tilde c_{d,\beta}={\Gamma(\beta)\over \pi^{d/2}\Gamma(\beta-d/2)},
\end{equation}
where $\|\,\cdot\,\|$ denotes the Euclidean norm in $\RR^d$. We let $\tilde P_{n,d}^\beta:={\rm conv}(\tilde X_1,\ldots,\tilde X_n)$ be the convex hull of $n\in\NN$ independent random points $\tilde X_1,\ldots,\tilde X_n$ distributed in $\RR^d$ according to the density $\tilde f_{d,\beta}$. This random polytope is known as a so-called \emph{beta' polytope}. In our notation we follow~\cite{kabluchko_formula,KabluchkoTemesvariThaele,beta_polytopes}, where these polytopes were studied.  As in the spherical case, we denote by $(f_0(P),f_1(P),\ldots,f_{d-1}(P))$ the $f$-vector of a polytope $P\subset\RR^d$, where $f_k(P)$, $k\in\{0,1,\ldots,d\}$, is the number of $k$-dimensional faces of $P$.
Our main result relates the $f$-vector of $\sV_{n+1,d}$ to that of $\tilde P_{n,d}^\beta$ with $\beta=d$  and can be formulated as follows. The proof is postponed to Section~\ref{sec:proof}.

\begin{theorem}\label{thm:fvector}
For each $n\geq d+1$ we have that
$$
\big(f_k(\sV_{n+1,d})\big)_{k=0}^{d-1}\overset{d}{=} \big(f_{d-k-1}(\tilde P_{n,d}^d)\big)_{k=0}^{d-1},
$$
where $\overset{d}{=}$ denotes equality in distribution of random vectors.
\end{theorem}

\subsection{Reduction to beta polytopes}
Recall that $X_1,\ldots,X_{n}$ are independent and uniformly distributed random points on $\SS^d$. Denote their convex hull  {in $\R^{d+1}$} by $P_{n, d+1}^{-1}:={\rm conv}(X_1,\ldots,X_n)$. This random polytope is a particular case of a \emph{beta polytope} with parameter $\beta = -1$ studied in~\cite{kabluchko_formula,KabluchkoTemesvariThaele,beta_polytopes}. We follow the notation used there. Our next theorem expresses the expected $f$-vector of $\sV_{n,d}$ in terms of that of $P_{n,d+1}^{-1}$.

\begin{theorem}\label{thm:fvector_beta}
For each {  $n\geq d+1$} and $k\in\{0,1,\ldots,d\}$ we have that
$$
\E f_k(\sV_{n,d}) = \frac{d-k+1}{n}\, \E f_{d-k}(P_{n,d+1}^{-1}).
$$
\end{theorem}
{ 
\begin{remark}
There is a duality between the faces of the spherical Voronoi tessellation $\mfm_{n,d}$ and the faces of the convex hull of $X_1,\ldots,X_n$, which was stated already in the work of Edelsbrunner and Nikitenko~\cite[pp.~3226--3227]{EdelsbrunnerNikitenko}. It says that for arbitrary $\ell\in \{0,\ldots, d\}$ and $1\leq i_0 < \ldots < i_\ell\leq n$, the convex hull of $X_{i_0},\ldots,X_{i_\ell}$ is a face of the convex hull of $X_1,\ldots,X_n$ if and only if the spherical Voronoi cells $C_{i_0,n}, \ldots, C_{i_\ell,n}$ have a non-empty intersection. This intersection is then a common face of these cells of dimension $d-\ell$, with probability $1$. In the proof given below, we provide a detailed  explanation of this duality based on~\cite[pp.~472--473]{SW}.
\end{remark}
\begin{proof}[Proof of Theorem~\ref{thm:fvector_beta}]
First of all, let us provide a general representation for the faces of the spherical Voronoi tessellation $\mfm_{n,d}$.
Take some $i\in \{1,\ldots,n\}$ and consider the cell
\begin{equation}\label{eq:def_C_i_n}
C_{i,n} = \SS^d \cap \{y\in \R^{d+1}: \langle y, X_i\rangle \geq \langle y, X_j\rangle \text{ for all } j\in \{1,\ldots,n\} \}.
\end{equation}
Here, $\langle\,\cdot\,,\,\cdot\,\rangle$ stands for the standard scalar product in $\RR^{d+1}$.
In order to represent the relative interiors of the faces of this cell, we need to turn some of the inequalities $\langle y, X_i\rangle \geq \langle y, X_j\rangle$ into equalities, while making the remaining  inequalities strict; see, e.g., \cite[7.2(e) on p.~135]{padberg_book}.  Thus, the relative interiors of the faces of $\mfm_{n,d}$ admit a representation of  the form
\begin{equation}\label{eq:def_S1}
S := \{y\in \SS^{d}: \langle y, X_{i_0}\rangle = \ldots = \langle y, X_{i_\ell}\rangle > \langle y, X_j\rangle \text{ for all } j\notin \{i_0,\ldots,i_\ell\}\}
\end{equation}
for some $\ell\in \{0,\ldots, n\}$ and  $1\leq i_0 < \ldots < i_\ell\leq n$.  Conversely, any set $S$ of the above form~\eqref{eq:def_S1} is a relative interior of a face of $\mfm_{n,d}$ provided $S\neq \varnothing$.  Observe that for $\ell > d$ the vectors  $X_{i_1}-X_{i_0}, \ldots, X_{i_\ell}-X_{i_0}$ linearly span $\R^{d+1}$ with probability $1$ and hence the only solution of $\langle y, X_{i_0}\rangle = \ldots = \langle y, X_{i_\ell}\rangle$ is $y=0$ (implying that $S= \varnothing$). Thus, we may assume that $\ell\in \{0,\ldots, d\}$. Then, the vectors $X_{i_1}-X_{i_0}, \ldots, X_{i_\ell}-X_{i_0}$ are linearly independent almost surely and hence the dimension of the set $S$ defined in~\eqref{eq:def_S1} is $d-\ell$, provided $S\neq \varnothing$.

Let us now provide a description of the faces of the polytope ${\rm conv}(X_1,\ldots,X_n)$. For $\ell\in \{0,\ldots, d\}$ and $1\leq i_0 < \ldots < i_\ell\leq n$ let $E$ be the  affine subspace through the points $X_{i_0},\ldots,X_{i_\ell}$. In the following, we exclude an event of probability $0$ and assume that $E$ is $\ell$-dimensional.  Let $E_0$ be the translate of $E$ passing through the origin of $\R^{d+1}$, that is, $E_0$ is the linear hull of $X_{i_1}-X_{i_0}, \ldots, X_{i_\ell}-X_{i_0}$.   Put
$$
F=E_0^\perp= \{y\in \RR^{d+1}: \langle y, X_{i_0}\rangle = \ldots = \langle y, X_{i_\ell}\rangle\}
$$
and note that $F$ is a linear subspace of dimension $d+1-\ell$. The intersection of $F$ with $\SS^d$ is a $(d-\ell)$-dimensional great subsphere of $\SS^d$ consisting of all points $y\in \SS^d$ having equal geodesic distances to $X_{i_0},\ldots,X_{i_\ell}$. For $y\in F\cap \SS^d$ write $r(y):=\rho(y, X_{i_0}) = \ldots =\rho(y,X_{i_\ell})$. Denote by ${\rm Cap}(y,r(y)):=\{z\in\SS^d:\rho(y,z) \leq r(y)\}$ the closed spherical cap centred at $y$ with geodesic radius $r(y)$ and put
\begin{equation}\label{eq:def_S2}
S :=\{y\in F\cap\SS^d:{\rm Cap}(y,r(y))\cap\{X_1,\ldots,X_n\}= \{X_{i_0},\ldots, X_{i_\ell}\}\},
\end{equation}
which is just an equivalent form of~\eqref{eq:def_S1}.

We now claim that $S\neq\varnothing$ if and only if ${\rm conv}(X_{i_0},\ldots,X_{i_\ell})$ is a face of   ${\rm conv} (X_1,\ldots,X_n)$. Indeed, if $S\neq \varnothing$, then there is $y\in \SS^d$ such that $a= \langle y, X_{i_0}\rangle = \ldots = \langle y, X_{i_\ell}\rangle$ and $\langle y, X_j\rangle > a$ for all indices $j\notin \{i_0,\ldots, i_\ell\}$. Consider the hyperplane $H:= \{z\in \R^{d+1}: \langle z,y \rangle = a\}$. Then, $H$ is a supporting hyperplane for ${\rm conv}(X_1,\ldots,X_n)$ and ${\rm conv} (X_{i_0},\ldots, X_{i_\ell}) = {\rm conv} (X_1,\ldots,X_n) \cap H$ is the corresponding face of ${\rm conv} (X_1,\ldots,X_n)$, thus proving the forward direction of the claim. To prove the backward direction, one assumes that ${\rm conv}(X_{i_0},\ldots,X_{i_\ell})$ is a face of ${\rm conv}(X_1,\ldots,X_n)$ corresponding to some supporting hyperplane $H$ which must be of the form $H:= \{z\in \R^{d+1}: \langle z,y \rangle = a\}$ for some $y\in \SS^d$ and $a\in \R$. It follows that  $a= \langle y, X_{i_0}\rangle = \ldots = \langle y, X_{i_\ell}\rangle$ and (without restriction of generality) $\langle y, X_j\rangle > a$ for all $j\notin \{i_0,\ldots, i_\ell\}$. This proves our claim. Although we shall not need it, notice also the following consequence of~\eqref{eq:def_S1} and~\eqref{eq:def_C_i_n}: the closure of $S$ coincides with $C_{i_0,n}\cap \ldots \cap C_{i_\ell,n}$.

Summarizing, we proved that there is a bijective correspondence between the $(d-\ell)$-dimensional faces of $\mfm_{n,d}$, the $\ell$-dimensional faces of ${\rm conv} (X_1,\ldots,X_n)$, and the tuples $1\leq i_0 < \ldots < i_\ell\leq n$ for which the set  $S$ defined in~\eqref{eq:def_S1} or~\eqref{eq:def_S2} is non-empty.  Thus, taking $\ell:= d-k$ we conclude that
\begin{equation}\label{eq:DualityTessPoly}
f_{d-k}(P_{n,d+1}^{-1}) = f_k(\mfm_{n,d}) \;\;\; \text{almost surely}.
\end{equation}
On the other hand, by Proposition~\ref{prop:total_cells} we have
$$
\EE f_k (\mfm_{n,d}) = \frac{n}{d-k+1}\,\EE f_k(\sV_{n,d}).
$$
Putting these results together completes the proof of Theorem~\ref{thm:fvector_beta}.
\end{proof}
} 

\section{Explicit formulas and special cases}

\subsection{Explicit formula for the expected \texorpdfstring{$f$}{f}-vector}

{
The expected $f$-vectors of beta and beta' polytopes have been explicitly determined in the series of works~\cite{kabluchko_algorithm,kabluchko_formula,kabluchko_angles,kabluchko_poisson_zero,beta_polytopes}. The main results we shall rely on are stated in Theorems 7.1 and 7.3 of~\cite{kabluchko_formula}.   Combining these formulae with Theorem~\ref{thm:fvector} or Theorem~\ref{thm:fvector_beta} we arrive at the following explicit expression for the $f$-vector of the typical Voronoi cell $\sV_{n+1,d}$.}

\begin{theorem}\label{thm:fvectorExplicit}
For all $d\geq 2$, $n\geq d+1$ and  $\ell\in \{1,\ldots,d\}$ we have
\begin{align}
\E f_{d-\ell} (\sV_{n+1,d})
&=
\frac 1 \pi  \left(\frac{\Gamma(\frac{d+1}{2})}{\sqrt \pi\, \Gamma(\frac d2)}\right)^{n-\ell}
\sum_{\substack{m\in \{\ell,\ldots,d\} \\ m\equiv d \Mod{2}}}
\tilde I_d(n,m) (md-1) \tilde J_d(m,\ell)\label{eq:f_vect_main1}
\\
&=
\frac 1 \pi  \left(\frac{\Gamma(\frac{d+1}{2})}{\sqrt \pi\, \Gamma(\frac d2)}\right)^{n-\ell}
\sum_{\substack{m\in \{\ell,\ldots,d\} \\ m\equiv d \Mod{2}}}
I_{d-1}(n,m) ((m+1)(d-1)+1) J_{d-1}(m,\ell) \label{eq:f_vect_main2}
,
\end{align}
where
\begin{align*}
\tilde I_d(n,m)
&:=
\binom{n}{m} \int_{-\pi/2}^{+\pi/2} (\cos x)^{d m -1} (\tilde F_d(x))^{n-m} \,\dint x,\\
\tilde J_d(m,\ell)
&:=
\binom{m}{\ell} \int_{-\infty}^{+\infty} (\cosh y)^{-d m + 1} (\tilde F_d(\ii y))^{m-\ell} \,\dint y,\\
I_{d-1}(n,m)
&:=
\binom{n}{m} \int_{-\pi/2}^{+\pi/2} (\cos x)^{(d-1)(m+1)} (F_{d-1}(x))^{n-m} \,\dint x,\\
J_{d-1}(m,\ell)
&:=
\binom{m}{\ell} \int_{-\infty}^{+\infty} (\cosh y)^{-(d-1)(m+1) - 2} (F_{d-1}(\ii y))^{m-\ell} \,\dint y,\\
\tilde F_d(z)
&=
F_{d-1}(z)
:=
\int_{-\pi/2}^z (\cos y)^{d-1} \,\dint y, \qquad  {  z\in \mathbb C}.
\end{align*}
\end{theorem}
{ 
\begin{remark}
Note that the imaginary unit $\ii= \sqrt{-1}$ appears because the $J$-quantities are related to the analytic continuation of the $I$-quantities~\cite{kabluchko_formula}. The integral in the definition of $\tilde F_d(z) =F_{d-1}(z)$ is taken along any contour connecting $-\pi/2$ and $z$.  Observe also that the quantities $\tilde J_d(m,\ell)$ and $J_{d-1}(m,\ell)$ are real-valued as one can see by making the substitution $y\mapsto -y$ in the defining integrals and using the fact that $\tilde F_d(\ii y)$ is the complex conjugate of $\tilde F_d(-\ii y)$, for all $y\in\R$.
\end{remark}
}
\begin{proof}[Proof of Theorem~\ref{thm:fvectorExplicit}]
We can give two proofs based on reduction of the spherical Voronoi tessellation to beta' and beta polytopes. These proofs yield~\eqref{eq:f_vect_main1} and~\eqref{eq:f_vect_main2}, respectively. Let us start with the approach based on beta' polytopes. By Theorem~\ref{thm:fvector}, we have
$$
\E f_{d-\ell} (\sV_{n+1,d}) = \E f_{\ell-1}(\tilde P_{n,d}^d).
$$
By~\cite[Theorem 7.3]{kabluchko_formula} applied with $\alpha = \beta =d$, we obtain
$$
\E f_{\ell-1}(\tilde P_{n,d}^d) = \frac {2\cdot n!} {\ell!}  \left(\frac{\Gamma(\frac{d+1}{2})}{d \sqrt \pi\, \Gamma(\frac d2)}\right)^{n-\ell}
\sum_{\substack{m\in \{\ell,\ldots,d\} \\ m\equiv d \Mod{2}}}
\tilde b\{n,m\} \left(m-\frac 1d\right) \tilde a\!\left[m-\frac{2}{d}, \ell - \frac 2 d\right],
$$
where
\begin{align*}
\tilde b\{n,m\} &=\frac {d^{n-m}}{(n-m)!} \int_{-\pi/2}^{+\pi/2} (\cos x)^{d m -1} (\tilde F_d(x))^{n-m} \;\dint x,\\
\tilde a\!\left[m-\frac{2}{d}, \ell - \frac 2 d\right] &= \frac {d^{m-\ell+1}}{(m-\ell)!} \cdot \frac 1 {2\pi} \int_{-\infty}^{+\infty} (\cosh y)^{-d m +1} (\tilde F_d(\ii y))^{m-\ell} \;\dint y,
\end{align*}
and where $\tilde F_d$ is as above. After straightforward transformations, we arrive at~\eqref{eq:f_vect_main1}.

On the other hand, we can give an alternative proof based on Theorem~\ref{thm:fvector_beta} which states that
$$
\E f_{d-\ell}(\sV_{n+1,d}) = \frac{\ell+1}{n+1} \, \E f_{\ell}(P_{n+1,d+1}^{-1}).
$$
The expected $f$-vector of the beta polytope is given explicitly in~\cite[Theorem~7.1]{kabluchko_formula}. Applying this theorem with $\beta=-1$ and $\alpha = d-1$, we obtain
\begin{multline*}
\E f_{\ell}(P_{n+1,d+1}^{-1}) = \frac{2\cdot (n+1)!}{(\ell+1)!} \left(\frac{\Gamma(\frac{d-1}{2})}{2\sqrt \pi \, \Gamma(\frac{d}{2})}\right)^{n-\ell} \\
\times \sum_{\substack{m\in \{\ell,\ldots,d\} \\ m\equiv d \Mod{2}}}
b\{n+1,m+1\} \left(m + 1 + \frac 1{d-1}\right)  a\!\!\left[m+1+\frac{2}{d-1}, \ell + 1 + \frac 2 {d-1}\right],
\end{multline*}
where
\begin{align*}
b\{n+1,m+1\} &=\frac {(d-1)^{n-m}}{(n-m)!} \int_{-\pi/2}^{+\pi/2} (\cos x)^{(d-1)(m+1)} (F_{d-1}(x))^{n-m} \,\dint x,\\
a\!\!\left[m+1 + \frac{2}{d-1}, \ell + 1 + \frac 2 {d-1}\right] &= \frac {(d-1)^{m-\ell+1}}{(m-\ell)!} \cdot \frac 1 {2\pi} \int_{-\infty}^{+\infty} (\cosh y)^{-(d-1)(m+1)-2} (F_{d-1}(\ii y))^{m-\ell} \,\dint y.
\end{align*}
After some transformations, we arrive at~\eqref{eq:f_vect_main2}.
\end{proof}

\begin{remark}
In particular, we obtained an indirect proof that the right-hand sides of~\eqref{eq:f_vect_main1} and~\eqref{eq:f_vect_main2} are equal. Finding a direct proof of this equality seems non-trivial. Let us also mention that, according to our numerical computations, the individual summands in~\eqref{eq:f_vect_main1} and~\eqref{eq:f_vect_main2} are, in general, not equal.
\end{remark}

\begin{proposition}
Let $d\geq 2$, $n\geq d+2$ and $k\in\{0,\ldots,d-1\}$. If $d$ is even, then $\E f_k(\sV_{n,d})$ is a rational number. If $d$ is odd, then $\E f_k(\sV_{n,d})$ is a linear combination of the numbers  $\pi^{-2r}$, where $r=0,1,\ldots, \lfloor \frac{n-d+k-1}{2}\rfloor$, with rational coefficients.
\end{proposition}
\begin{proof}
This follows from Theorem~\ref{thm:fvector_beta} together with~\cite[Theorem 7.2]{kabluchko_formula}. The same result could be deduced by combining Theorem~\ref{thm:fvector} with \cite[Theorem 7.4]{kabluchko_formula}.
\end{proof}

\begin{remark}
Along with the Voronoi tessellation it is natural to consider the so-called spherical hyperplane tessellation which is defined as follows.
As before, let $X_1,\ldots,X_n$ be $n$ independent, uniformly distributed random points on $\SS^d$, where $n\geq d+1$. Let $X_i^{\bot} = \{z\in\R^{d+1}: \langle z, X_i\rangle=0\}$ be the hyperplane orthogonal to $X_i$.
The hyperplanes $X_1^{\bot},\ldots, X_n^{\bot}$ dissect the sphere $\SS^d$ into spherical polytopes which constitute the \textit{spherical hyperplane tessellation}. The \textit{spherical Crofton cell} $\mathcal Z_{n,d}$ is defined as the almost surely unique cell of this tessellation that contains the north pole $e$. We have $\mathcal Z_{n,d} = \SS^d\cap (G_1\cap \ldots \cap G_n)$, where $G_i$ is the half-space bounded by $X_i^{\bot}$ and containing the north pole $e$. The expected $f$-vector of the spherical Crofton cell $\mathcal Z_{n,d}$ can be computed as follows. We observe that the dual of the convex cone $G_1\cap \ldots\cap G_n$ is the positive hull $D_n:=\pos(X_1^{-},\ldots,X_{n}^-)$ of the points $X_i^-:= - X_i \cdot \sgn \langle X_i,e\rangle$. The points $X_1^-,\ldots,X_n^-$ are independent and uniformly distributed on the lower half-sphere $\SS^d_-:= \{z\in \SS^d: \langle z,e \rangle \leq 0\}$. The corresponding $f$-vectors satisfy
$$
\E f_k (\mathcal Z_{n,d}) = \E f_{k+1} (G_1\cap \ldots\cap G_n) = \E f_{d-k} (D_n)
=
 \E f_{d-k-1} (\SS^d\cap D_n)
$$
for all $k\in \{0,\ldots,d-1\}$.
The expected face numbers of the random spherical polytopes $\SS^d\cap D_n$ that appear on the right-hand side have been explicitly computed in~\cite{kabluchko_poisson_zero}. These polytopes are also closely related to the beta' polytopes $\tilde P_{n,d}^{\beta}$, but this time with $\beta= \frac{d+1}{2}$, see~\cite{KabluchkoMarynychTemesThaele,kabluchko_poisson_zero}.
\end{remark}

\subsection{Low-dimensional cases}
Let us consider the low-dimensional cases separately. For example, in dimension $d=2$, if we take $\ell=1$ in Theorem~\ref{thm:fvectorExplicit} we arrive at the following result of Miles~\cite{MilesSphere}.
\begin{corollary}
For $d=2$ and $n\geq 3$ we have
\begin{equation}\label{eq:f_vect_d_2_miles}
\E f_0(\sV_{n+1,2}) = \E f_1(\sV_{n+1,2}) = 6 \cdot \frac{n-1}{n+1} = 6\Big(1-{2\over n+1}\Big) {.}
\end{equation}
\end{corollary}
\begin{proof}
According to Theorem \ref{thm:fvectorExplicit}  we have that $\E f_1(\sV_{n+1,2})={1\over \pi}{1\over 2^{n-1}}\tilde{I}_2(n,2)\cdot 3\cdot\tilde{J}_2(2,1)$. Moreover, $\tilde{F}_2(z)=1+\sin z$, which implies that $\tilde{J}_2(2,1)=\pi$. In addition,
\begin{align*}
{n\choose 2}^{-1}\tilde{I}_2(n,2)
&= \int_{-\pi/2}^{\pi/2}(\cos x)^3(1+\sin x)^{n-2}\,\dint x
=
\sum_{k=0}^{n-2}    {\binom{n-2}{k}}      \int_{-\pi/2}^{\pi/2}(\cos x)^3(\sin x)^k\,\dint x\\
&=
\sum_{k=0}^{n-2}{\binom{n-2}{k}}   {2(1+(-1)^k)\over (k+1)(k+3)} = {2^{n+1}\over n(n+1)},
\end{align*}
which implies $\E f_1(\sV_{n+1,2})={3\over 2^{n-1}}{n\choose 2}{2^{n+1}\over n(n+1)}=6\cdot{n-1\over n+1}$.
\end{proof}

{As observed already by Miles~\cite{MilesSphere}}, it is not surprising that $\E f_0(\sV_{n+1,2})\to 6$, as $n\to\infty$, which is the expected number of vertices of the typical cell of a Poisson-Voronoi tessellation in the plane, see \cite[Theorem 10.2.5]{SW}.

\begin{remark}
We note that~\eqref{eq:f_vect_d_2_miles} can alternatively be obtained by purely combinatorial means. Indeed, by the Euler relation and the fact that the Voronoi tessellation on the sphere is almost surely simple (which, for $d=2$, means that each vertex of the tessellation belongs to exactly $3$ edges), we have
$$
 f_2(\mfm_{n+1,d}) - f_1(\mfm_{n+1,d}) + f_0(\mfm_{n+1,d}) = 2 \;\;\text{ and }\;\; 2 f_1(\mfm_{n+1,d}) = 3 f_0(\mfm_{n+1,d})\;\;\; \text{almost surely}.
$$
Also, $f_2(\mfm_{n+1,d}) = n+1$ since each cell corresponds to its centre. Altogether, it follows that
$$
f_0(\mfm_{n+1,d}) = 2(n-1)\;\; \text{ and }\;\; f_1(\mfm_{n+1,d}) = 3(n-1) \;\; \text{almost surely}.
$$
Taking the expectations and recalling Proposition~\ref{prop:total_cells} yields~\eqref{eq:f_vect_d_2_miles}.
\end{remark}

\begin{table}
\begin{tabular}{|c||c|c|c|}
\hline
\parbox[0pt][2em][c]{0cm}{} & $n=4$ & $n=5$& $n=6$\\
 \hline
\parbox[0pt][2em][c]{0cm}{}$\E f_0(\sV_{n+1,3})$ & $4$ & ${20\over 3}-{3\,289\over 360\pi^2}\approx 5.74$ & $10-{3\,289\over 120\pi^2}\approx 7.22$ \\
\hline
\parbox[0pt][2em][c]{0cm}{}$\E f_0(\sV_{n+1,4})$ & -- & $5$ & ${18975\over 2261}\approx 8.39$\\
\hline
\hline
\parbox[0pt][2em][c]{0cm}{}& $n=7$& $n=8$& $n=9$\\
\hline
\parbox[0pt][2em][c]{0cm}{}$\E f_0(\sV_{n+1,3})$ & {\tiny $14-{23\,023\over 360\,\pi^2}+{569\,556\,559\over 6\,048\,000\,\pi^4}$} & {\tiny ${56\over 3}-{23\,023\over 180\,\pi^2}+{569\,556\,559\over 1\,512\,000\,\pi^4}$} & {\tiny $24-{23\,023\over 100\,\pi^2}+{569\,556\,559\over 504\,000\,\pi^4}-{200\,082\,581\,646\,233\over 118\,540\,800\,000\,\pi^6}$}\\
& $\;\;\;\;\;\,\approx 8.49$ & $\;\;\;\;\;\;\;\;\,\approx 9.57$ & $\;\;\;\;\;\;\;\,\approx 10.52$\\
\hline
\parbox[0pt][2em][c]{0cm}{}$\E f_0(\sV_{n+1,4})$ & ${3835\over 323}\approx 11.87$ & ${340886\over 22287}\approx 15.29$ & ${8124\over 437}\approx 18.59$ \\
\hline
\end{tabular}
\caption{Exact and approximate values for the expected number of vertices of  {the} typical Voronoi cell generated by $n\in\{4,\ldots,9\}$ random points on $\SS^3$ and $n\in\{5,\ldots,9\}$ random points on $\SS^4$.}
\label{tab:d=3}
\end{table}

On the other hand, in dimensions $d>2$ the $f$-vector of $\mfm_{n+1,d}$ is not deterministic.  For $d=3$ and $d=4$, we present exact formulae for the expected $f$-vector of the typical spherical Voronoi cell and refer to Table~\ref{tab:d=3} for some exact and numerical values for small values of $n$.

\begin{corollary}
For $d=3$ and all $n\geq 4$ we have
\begin{align*}
\E f_0(\sV_{n+1,3})
&=
{256\over 35\pi}\Big({1\over 2\pi}\Big)^{n-3}\binom n3  \int_{-\pi/2}^{+\pi/2}  (\cos x)^8 (2 x+\sin (2 x)+\pi )^{n-3} \, \dint x,\\
\E f_1(\sV_{n+1,3})
&=
\frac 32\, \E f_0(\sV_{n+1,3}),\\
\E f_2(\sV_{n+1,3})
&=
\frac 12\, \E f_0(\sV_{n+1,3}) + 2.
\end{align*}
\end{corollary}
\begin{proof}
The first formula follows from Theorem~\ref{thm:fvectorExplicit} with $d=3$ and $\ell=3$:
$$
\E f_0(\sV_{n+1,3}) = {1\over\pi}\Big({2\over\pi}\Big)^{n-3}\tilde{I}_3(n,3)\cdot 8 \cdot\tilde{J}_3(3,3).
$$
It remains to note that $\tilde{F}_3(z)={1\over 4}(2z+\sin(2z)+\pi)$, which implies that $\tilde{J}_3(3,3)={32\over 35}$ and
$$
\tilde{I}_3(n,3) = \Big({1\over 4}\Big)^{n-3}{n\choose 3}\int_{-\pi/2}^{+\pi/2}(\cos x)^8(2x+\sin(2x)+\pi)^{n-3}\,\dint x.
$$
Since $\sV_{n+1,3}$ is a simple polytope with probability one, we have that almost surely $2f_1(\sV_{n+1,3})=3f_0(\sV_{n+1,3})$. Finally, the formula for $\E f_0(\sV_{n+1,3})$ follows from Euler's relation, which says that almost surely $f_0(\sV_{n+1,3})-f_1(\sV_{n+1,3})+f_2(\sV_{n+1,3})=2$.
\end{proof}

\begin{corollary}
For $d=4$ and all $n\geq 5$ we have
\begin{align*}
\E f_0(\sV_{n+1,4})
&=
\frac{6435}{2048}\Big({3\over 48}\Big)^{n-4} \binom n 4 \int_{-\pi/2}^{+\pi/2}  (\cos x)^{15} (8+ 9\sin x+\sin (3 x))^{n-4} \, \dint x,\\
\E f_1(\sV_{n+1,4})
&=2\E f_0(\sV_{n+1,4}),\\
\E f_2(\sV_{n+1,4})
&=6\,{n-1\over n+1}+{6\over 5}\,\E f_0(\sV_{n+1,4}),\\
\E f_3(\sV_{n+1,4})
&=6\,{n-1\over n+1}+{1\over 5}\,\E f_0(\sV_{n+1,4}).
\end{align*}
\end{corollary}
\begin{proof}
The identity for $\E f_0(\sV_{n+1,4})$ follows from Theorem \ref{thm:fvectorExplicit}. In fact, taking $d=4$ and $\ell=4$ we obtain
$$
\E f_0(\sV_{n+1,4}) = {1\over\pi}\Big({3\over 4}\Big)^{n-4}\tilde{I}_4(n,4)\cdot 15\cdot\tilde{J}_4(4,4).
$$
Moreover, $\tilde{F}_4(z)={1\over 12}(8+9\sin z+\sin(3z))$, which in turn implies that $\tilde{J}_4(4,4)={429\pi\over 2048}$ and
$$
\tilde{I}_4(n,4) = \Big({1\over 12}\Big)^{n-4}{n\choose 4}\int_{-\pi/2}^{+\pi/2}(\cos x)^{15}(8+9\sin x+\sin(3x))^{n-4}\,\dint x .
$$
To derive the other identities, we use the $3$ linearly independent Dehn-Sommerville equations for simplicial $5$-dimensional polytopes \cite[Corollary 17.8]{BronstedtBook}. Applied to $P_{n+1,5}^{-1}$ they say that almost surely
\begin{align*}
2 & = f_0(P_{n+1,5}^{-1}) - f_1(P_{n+1,5}^{-1}) + f_2(P_{n+1,5}^{-1}) - f_3(P_{n+1,5}^{-1}) + f_4(P_{n+1,5}^{-1}),\\
2f_1(P_{n+1,5}^{-1}) &= 3f_2(P_{n+1,5}^{-1})-6f_3(P_{n+1,5}^{-1})+10f_4(P_{n+1,5}^{-1}),\\
5f_4(P_{n+1,5}^{-1}) &= 2f_3(P_{n+1,5}^{-1}).
\end{align*}
Using \eqref{eq:DualityTessPoly} these identities translate into the almost sure relations
\begin{align*}
2 &= f_4(\mfm_{n+1,4}) - f_3(\mfm_{n+1,4}) + f_2(\mfm_{n+1,4}) - f_1(\mfm_{n+1,4}) + f_0(\mfm_{n+1,4}),\\
2f_3(\mfm_{n+1,4}) &= 3f_2(\mfm_{n+1,4}) - 6f_1(\mfm_{n+1,4}) + 10f_0(\mfm_{n+1,4}),\\
5f_0(\mfm_{n+1,4}) &= 2f_1(\mfm_{n+1,4})
\end{align*}
for the random Voronoi tessellation $\mfm_{n+1,4}$ on $\SS^4$. In addition, we have that almost surely $f_4(\mfm_{n+1,4})=n+1$, since each cell of $\mfm_{n+1,4}$ corresponds to its centre. This implies that $f_1(\mfm_{n+1,4})$, $f_2(\mfm_{n+1,4})$ and $f_3(\mfm_{n+1,4})$ can be expressed in terms of $f_0(\mfm_{n+1,4})$ only. In fact, we have that almost surely
\begin{align*}
f_1(\mfm_{n+1,4}) &= {5\over 2}f_0(\mfm_{n+1,4}),\\
f_2(\mfm_{n+1,4}) &= 2(n-1) + 2f_0(\mfm_{n+1,4}),\\
f_3(\mfm_{n+1,4}) &= 3(n-1) + {1\over 2}f_0(\mfm_{n+1,4}).
\end{align*}
We finally apply Proposition \ref{prop:total_cells} to conclude that $5\E f_0(\mfm_{n+1,4})=(n+1)\E f_0(\sV_{n+1,4})$ and
\begin{align*}
\E f_1(\sV_{n+1,4}) &= {4\over n+1}\,\E f_1(\mfm_{n+1,4}) = {4\over n+1}\cdot{5\over 2}\cdot \E f_0(\mfm_{n+1,4}) = 2\E f_0(\sV_{n+1,4}).
\end{align*}
The identities for $\E f_2(\sV_{n+1,4})$ and $\E f_3(\sV_{n+1,4})$ follow similarly:
\begin{align*}
\E f_2(\sV_{n+1,4}) &= {3\over n+1}\big(2(n-1)+2\E f_0(\mfm_{n+1,4})\big)=6\,{n-1\over n+1}+{6\over 5}\,\E f_0(\sV_{n+1,4}),\\
\E f_3(\sV_{n+1,4}) &= {2\over n+1}\Big(3(n-1)+{1\over 2}\E f_0(\mfm_{n+1,4})\Big)=6\,{n-1\over n+1} + {1\over 5}\,\E f_0(\sV_{n+1,4}).
\end{align*}
This completes the argument.
\end{proof}


\begin{remark}
It is interesting to note that if we would apply the  Dehn-Sommerville equations directly to the typical Voronoi cell $\sV_{n+1,4}$ (which is almost surely a simple polytope), this would not yield enough relations to express all $\E f_i(\sV_{n+1,4})$ through $\E f_0(\sV_{n+1,4})$.  {It is known~\cite[\S17]{BronstedtBook} that both, in dimensions $4$ and $5$, the $f$-vectors of simplicial (and simple) polytopes depend on $2$ free parameters. Applying the Dehn-Sommerville relations to the $5$-dimensional beta polytope has the advantage that we know the number of vertices to be $n+1$, which reduces the number of free parameters to $1$.}
\end{remark}

\section{Proof of Theorem \ref{thm:fvector}}\label{sec:proof}
\subsection{Preliminaries}
Let us first introduce some notation. Recall that  $\xi_n=\{X_1,\ldots,X_n\}$ is a binomial process on $\SS^d$ induced by $n\in\NN$ independent random points $X_1,\ldots,X_n$ with the uniform distribution $\sigma_d$. For each $i\in\{1,\ldots,n\}$ we let $h_i\in[-1,1]$ be the projection of $X_i$ onto the $0$-th coordinate of $\R^{d+1}$ which is shown as the vertical direction in Figure \ref{fig:proof}.  {Also, we denote by $\theta_i\in [0,\pi]$ the angle between $e=(1, 0,\ldots,0)$ and $X_i$. Formally,
$$
h_i = \langle X_i,e\rangle = \cos \theta_i,
$$
where $\langle\,\cdot\,,\,\cdot\,\rangle$ denotes the standard scalar product in $\R^d$.}
We can then decompose $X_i$ as follows:
\begin{equation}\label{eq:X_i_rep}
X_i =  e \cos \theta_i + U_i \sin \theta_i, \qquad i\in \{1,\ldots,n\},
\end{equation}
where $U_i$ is a suitable unit vector in the $d$-dimensional hyperplane $e^{\bot} = \{x_{0} = 0\}$ which we identify with $\R^d$.
{  The next lemma, which is well known, characterizes the joint distribution of  $h_i$ and $U_i$. It is just a probabilistic restatement of the slice integration formula for spheres, see Corollary A.5 in~\cite{AxlerBook}.
The density of $h_i$ can be found in~\cite[Lemma 7.6]{KabluchkoMarynychTemesThaele} or deduced from~\cite[Lemma 4.4]{KabluchkoTemesvariThaele}. 

\begin{lemma}\label{lem:DensityG}
For each $i\in\{1,\ldots,n\}$ the random variable $h_i$ has density
\begin{equation*}
f(h) = {\Gamma({d+1\over 2})\over\sqrt{\pi}\,\Gamma({d\over 2})}\,(1-h^2)^{{d\over 2}-1},\qquad h\in[-1,1],
\end{equation*}
with respect to the Lebesgue measure on $[-1,1]$. The random variable $U_i$ is uniformly distributed on the unit sphere in $\{x_{0} = 0\}\equiv \R^d$.  Finally, $h_i$ and $U_i$ are independent.
\end{lemma}
} 

\begin{figure}[t]
\begin{center}
\begin{tikzpicture}[scale=1.4]
\draw[very thick] (0,0) circle (2);
\draw[very thick] (-3,2) -- (3,2);
\draw (0,0) -- (0,2);
\draw (0,0) -- (1.8,0.871);
\draw[dashed] (1.8,0.871) -- (0,0.871);
\draw (0,0) -- (1.6,2.5);
\draw (0,0) -- (-1.6,-2.5);

\fill[black] (0,2)  circle(0.075) node[above] {$e$};
\fill[black] (0,0)  circle(0.075) node[left] {$0$};
\fill[black] (1.8,0.871)  circle(0.075) node[right] {$X_i$};
\fill[black] (0,0.871)  circle(0.075) node[left] {$h_i$};
\node at (0.7,2.25) {$R_i$};
\node at (-2.25,2.25) {\small $\{x_0=1\}$};
\node at (1.9,2.5) {$L_i$};
\node at (-2,-1.15) {$\SS^d$};
\node at (0.75,-0.25) {$\theta_i\over 2$};
\node at (-0.5,0.4) {$\theta_i\over 2$};
\draw[->] (0.6,-0) -- (0.25,0.25);
\draw[->] (-0.3,0.4) -- (0.15,0.4);


\end{tikzpicture}
\end{center}
\label{fig:proof}
\caption{Illustration of the construction used in the proof of Theorem \ref{thm:fvector}.}
\end{figure}
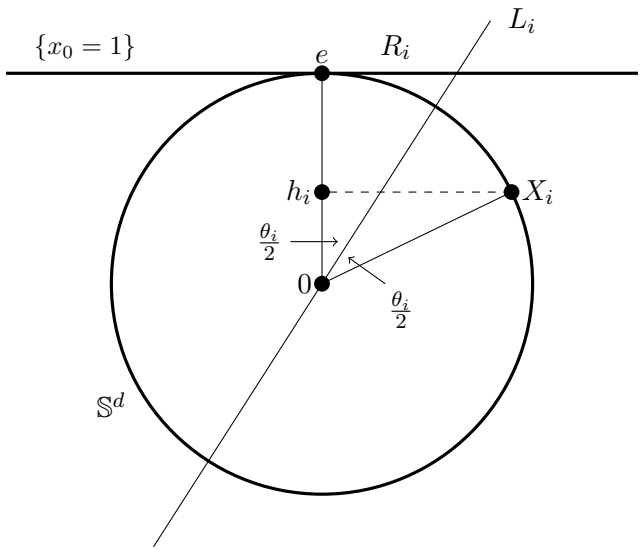

\subsection{Proof of Theorem \ref{thm:fvector}}
The starting point of our proof is the representation of the typical Voronoi cell on $\SS^d$ given in Proposition~\ref{prop:typical_voronoi_rep}:
\begin{equation*}
\sV_{n+1,d} \eqdistr \bigcap_{i=1}^n  \{z\in\SS^d: \rho(e,z)\leq \rho(X_i,z)\}.
\end{equation*}
Recalling that the geodesic distance on $\SS^d$ is given by $\rho(x,y) = \arccos \langle x,y\rangle$, $x,y\in \SS^d$, and using that the function $u\mapsto \arccos u$ is decreasing on $[-1,1]$, we can write the above representation as
\begin{equation}\label{eq:voronoi_as_intersection}
\sV_{n+1,d} \eqdistr  \bigcap_{i=1}^n (L_i^+ \cap \SS^d),
\end{equation}
where $L_1^+,\ldots,L_n^+\subset \R^{d+1}$ are half-spaces defined by
$$
L_i^+ := \{z\in\R^{d+1}: \langle e,z\rangle \geq \langle X_i,z\rangle\}
=
\{z\in\R^{d+1}: \langle X_i-e,z\rangle \leq 0 \}, \qquad i\in\{1,\ldots,n\}.
$$
The bounding hyperplane of $L_i^+$ is denoted by
$$
L_i :=
\{z\in\R^{d+1}: \langle X_i-e,z\rangle = 0 \}, \qquad i\in\{1,\ldots,n\},
$$
{  see Figure \ref{fig:proof}.} 
Note that $L_i$ passes through the origin of $\R^{d+1}$ and that $e\in L_i^+$. Consider the convex random polyhedral cone
$$
C_n := \bigcap_{i=1}^n L_i^+ \subset \R^{d+1},
$$
{  see Figure \ref{fig:proof2}.}
By definition, the $k$-dimensional faces of the spherical polytope $C_n \cap \SS^d$ (which is the right-hand side of~\eqref{eq:voronoi_as_intersection}) are in bijective correspondence with the  $(k+1)$-dimensional  faces of the polyhedral cone $C_n$. Thus, we arrive at the distributional equality
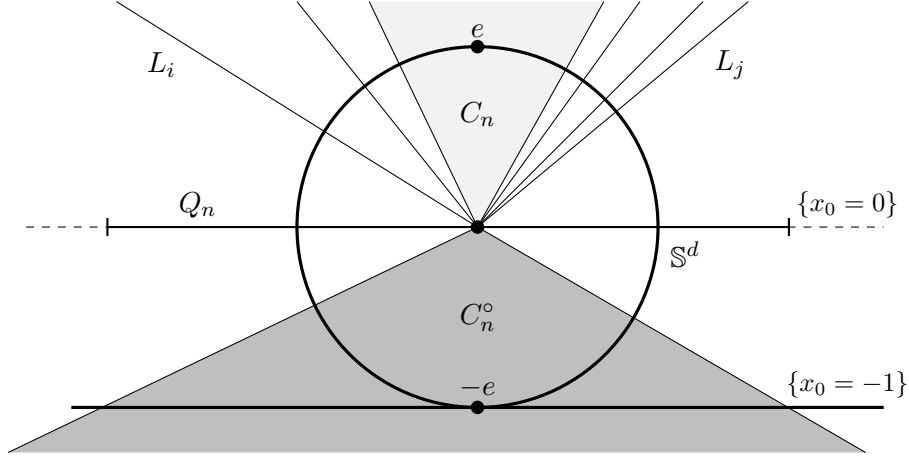
\begin{figure}[t]
\begin{center}
\begin{tikzpicture}[scale=1.2]

\fill[fill=gray!10] (0+8,0)--(1.4+8,2.5)--(-1.2+8,2.5);
\fill[fill=gray!50] (0+8,0)--(-5.2+8,-2.5)--(4.3+8,-2.5);
\draw[very thick] (0+8,0) circle (2);
\draw[very thick] (-4.5+8,-2) -- (4.5+8,-2);
\draw (0+8,0) -- (1.4+8,2.5);
\draw (0+8,0) -- (1.8+8,2.5);
\draw (0+8,0) -- (3+8,2.5);
\draw (0+8,0) -- (-1.2+8,2.5);
\draw (0+8,0) -- (-2+8,2.5);
\draw (0+8,0) -- (-4+8,2.5);
\draw (0+8,0) -- (+2.5+8,+2.5);
\draw (0+8,0) -- (4.3+8,-2.5);
\draw (0+8,0) -- (-5.2+8,-2.5);
\draw[thick] (-4.1+8,0) -- (3.45+8,0);
\draw[dashed] (-5+8,0) -- (4.5+8,0);
\draw[thick] (-4.1+8,-0.1) -- (-4.1+8,0.1);
\draw[thick] (3.45+8,-0.1) -- (3.45+8,0.1);

\node at (+2.3+8,-0.3) {$\SS^d$};
\fill[black] (0+8,2)  circle(0.075) node[above] {$e$};
\fill[black] (0+8,-2)  circle(0.075) node[above] {$-e$};
\node at (0+8,1.25)  {$C_n$};
\node at (0+8,-1)  {$C_n^\circ$};
\fill[black] (0+8,0.0)  circle(0.075);
\node at (4.1+8,-1.75) {\small $\{x_0=-1\}$};
\node at (4.1+8,0.25) {\small $\{x_0=0\}$};
\node at (-3.1+8,0.25) { $Q_n$};
\node at (2.8+8,1.8) {$L_j$};
\node at (-3.5+8,1.8) {$L_i$};
\end{tikzpicture}
\end{center}
\label{fig:proof2}
\caption{Illustration of the cones $C_n$ and $C_n^\circ$ as well as the random polytope $Q_n$.}
\end{figure}
\begin{equation}\label{eq:f_vect_proof_rep}
\big(f_k(\sV_{n+1,d})\big)_{k=0}^{d-1}
\eqdistr
\big(f_{k}(C_n \cap \SS^d)\big)_{k=0}^{d-1}
=
\big(f_{k+1}(C_n)\big)_{k=0}^{d-1}.
\end{equation}
The \textit{dual} or \textit{polar} of the convex cone $C_n$ is defined as
$$
C_n^\circ := \{x\in\RR^{d+1}:\langle x,y\rangle\leq 0\text{ for all }y\in C_n\}.
$$
Since the $(k+1)$-dimensional faces of $C_n$ are in bijective correspondence with the $(d-k)$-dimensional faces of $C_n^\circ$, it follows from~\eqref{eq:f_vect_proof_rep} that
\begin{equation}\label{eq:f_vect_proof_rep1}
\big(f_k(\sV_{n+1,d})\big)_{k=0}^{d-1}  \eqdistr  \big(f_{d-k}(C_n^\circ)\big)_{k=0}^{d-1}.
\end{equation}
Since $C_n$ is defined as the intersection of the half-spaces $L_1^+,\ldots,L_n^+$, the dual cone is the positive hull of the outward normal vectors of these half-spaces, that is
$$
C_n^\circ = \pos (X_1-e,\ldots,X_n-e).
$$
{  Since the $0$-th coordinates of the vectors $X_i-e$ are strictly negative almost surely, it follows that $C_n^\circ\backslash\{0\}$ is contained  $\{x_0 < 0\}$ almost surely.}
Recall from \eqref{eq:X_i_rep} that $X_i - e = e (\cos \theta_i-1) + U_i \sin \theta_i$. We may ignore the case when some $\theta_i=0$ because it has probability $0$.
Normalizing the vectors spanning $C_n^\circ$ in such a way that their $0$-th coordinate becomes $-1$, we get
$$
C_n^\circ = \pos \left(-e + \frac{U_1}{R_1},\ldots,-e + \frac{U_n}{R_n}\right),
$$
where
$$
R_i := \frac{1 - \cos \theta_i}{\sin \theta_i} = \tan\Big({\theta_i\over 2}\Big), \qquad i\in \{1,\ldots,n\}.
$$
{  It follows from $C_n^\circ  \backslash\{0\}\subset \{x_0 <  0\}$ that} almost surely  the $(d-k)$-dimensional faces of $C_n^{\circ}$ are in one-to-one correspondence with the $(d-k-1)$-dimensional faces of the polytope obtained by intersecting $C_n^\circ$ with the tangent space to $\SS^d$ at its south pole $-e$. Define the polytope
\begin{equation}\label{eq:Q_n_def}
Q_n := (C_n^\circ \cap \{x_0 = -1\}) + e = {\rm conv}\left(\frac{U_1}{R_1},\ldots, \frac{U_n}{R_n}\right)\subset \{x_0=0\}.
\end{equation}
Recalling~\eqref{eq:f_vect_proof_rep1} we can write
\begin{equation}\label{eq:f_vect_proof_rep2}
\big(f_k(\sV_{n+1,d})\big)_{k=0}^{d-1}  \eqdistr  \big(f_{d-k-1}(Q_n)\big)_{k=0}^{d-1}.
\end{equation}



To complete the proof of Theorem~\ref{thm:fvector}, it remains to verify that the random polytope $Q_{n}$ has the same distribution as the beta' polytope $\tilde P_{n,d}^d$ in $\RR^d$ with parameter $\beta=d$. 

{ 
\begin{lemma}\label{lem:DensityH}
For each $i\in\{1,\ldots,n\}$ the random variable $R_i$ has density
\begin{equation}\label{eq:DensityR}
g(r) = {2^d\,\Gamma({d+1\over 2})\over\sqrt{\pi}\,\Gamma({d\over 2})}\,{r^{d-1}\over (1+r^2)^d},\qquad r\geq 0,
\end{equation}
with respect to the Lebesgue measure on $[0,\infty)$. Also, we have  $R_i\overset{d}{=} 1/R_i$.
\end{lemma}
\begin{proof}
The identity $R_i = \frac {1-\cos \theta_i} {\sin \theta_i}=\frac {\sin \theta_i}{1+\cos \theta_i}$ implies that
\begin{equation}\label{eq:R_i}
R_i^2 = \frac{1-\cos\theta_i}{\sin \theta_i} \cdot \frac{\sin \theta_i}{1+\cos\theta_i}  =   \frac{1-\cos\theta_i}{1+\cos\theta_i} = \frac {1-h_i}{1+h_i}.
\end{equation}
Since $h_i$ has the same distribution as $-h_i$ by Lemma~\ref{lem:DensityG}, we have $R_i\overset{d}{=} 1/R_i$. 
Furthermore,  for each $r\geq 0$,
\begin{align*}
\PP[R_i\geq r] &= \PP\Big[\sqrt{1-h_i\over 1+h_i}\geq r\Big]= \PP\Big[h_i\leq{1-r^2\over 1+r^2}\Big]= {\Gamma({d+1\over 2})\over\sqrt{\pi}\,\Gamma({d\over 2})}\int_0^{1-r^2\over 1+r^2}(1-h^2)^{{d\over 2}-1}\,\dint h,
\end{align*}
where the last identity comes from Lemma \ref{lem:DensityG}. Differentiation with respect to $r$ thus proves that the density of $R_i$ is
\begin{align*}
g(r) &= {\Gamma({d+1\over 2})\over\sqrt{\pi}\,\Gamma({d\over 2})}\,\bigg(1-\Big({1-r^2\over 1+r^2}\Big)^2\bigg)^{{d\over 2}-1}\,{4r\over(1+r^2)^2}= {2^d\,\Gamma({d+1\over 2})\over\sqrt{\pi}\,\Gamma({d\over 2})}\,{r^{d-1}\over (1+r^2)^d},
\end{align*}
which completes the argument.
\end{proof}

We are now in position to complete the  proof of Theorem~\ref{thm:fvector}.
Recall from Lemma~\ref{lem:DensityG} that $U_1,\ldots,U_n$ are i.i.d.\ and uniformly distributed on the unit sphere in $\R^d$. The same Lemma~\ref{lem:DensityG} (see also~\eqref{eq:R_i}) states that this family is independent of the collection $R_1,\ldots,R_n$  of random variables which are also i.i.d.\ and have density $g(r)$ given by \eqref{eq:DensityR}.
Altogether, recalling~\eqref{eq:Q_n_def}, it follows that
$$
Q_{n} \overset{d}{=} {\rm conv}\left(U_1 R_1,\ldots, U_nR_n\right)
$$
and that $U_1R_1,\ldots,U_nR_n$ are independent random points in $\R^d$ with Lebesgue density 
$$
\tilde f_{d, d}(x) := {\Gamma(d)\over \pi^{d/2}\Gamma(d/2)} \,
(1+\|x\|^2)^{-d},\qquad x\in\R^d.
$$ 
(The value of the constant follows from the Legendre duplication formula but is actually not needed for the argument).  This is the beta' density from~\eqref{eq:BetaPrimeDensity} with  $\beta=d$. 
Hence, $Q_{n}$ has the same distribution as the beta' polytope $\tilde P_{n,d}^d$, and the proof of Theorem \ref{thm:fvector} is complete. \hfill $\Box$

\begin{remark}
As a byproduct of the above proof, note that the beta' distribution  with $\beta=d$  stays  invariant  under inversion with respect to the unit sphere.
\end{remark}

}

\subsection*{Acknowledgement}
We would like to thank {  the unknown referee for numerous useful comments which substantially improved the presentation and to}  Maike Buchin (Bochum) for pointing us to the usage of Voronoi diagrams on manifolds in computational geometry. ZK has been supported by the German Research Foundation under Germany's Excellence Strategy  EXC 2044 -- 390685587, Mathematics M\"unster: Dynamics - Geometry - Structure.



\end{document}